\documentclass[a4paper, 12pt]{amsart}

\usepackage[T1]{fontenc}
\usepackage[utf8]{inputenc}

\usepackage{amssymb}
\usepackage{mathrsfs}
\usepackage{mathtools}
\usepackage{enumitem}

\usepackage{xparse}

\newcommand{\ov}[1]{\overline{#1}}

\allowdisplaybreaks[2]

\newtheorem{theorem}{Theorem}[section]
\newtheorem{proposition}[theorem]{Proposition}
\newtheorem{lemma}[theorem]{Lemma}

\theoremstyle{definition}
\newtheorem{remark}[theorem]{Remark}

\numberwithin{equation}{section}

\newtheorem{mainthm}{Theorem}

\newtheorem{conjecture}[mainthm]{Conjecture}

\newcommand{\N}{\mathbb N}

\newcommand{\Z}{\mathbb Z}
\newcommand{\C}{\mathbb C}

\newcommand{\HH}{\mathbb H}

\newcommand{\aaa}{\textbf{a}}
\newcommand{\bbb}{\textbf{b}}
\newcommand{\ccc}{\textbf{c}}
\newcommand{\ddd}{\textbf{d}}
\newcommand{\eee}{\textbf{h}}
\newcommand{\nnn}{\textbf{n}}
\newcommand{\rrr}{\textbf{r}}
\newcommand{\uuu}{\textbf{u}}
\newcommand{\vvv}{\textbf{v}}
\newcommand{\www}{\textbf{w}}
\newcommand{\xxx}{\textbf{x}}
\newcommand{\yyy}{\textbf{y}}
\newcommand{\zzz}{\textbf{z}}
\newcommand{\veca}{\vec{a}}
\newcommand{\vecb}{\vec{b}}

\newcommand{\vecd}{\vec{d}}
\newcommand{\vece}{\vec{e}}
\newcommand{\vecu}{\vec{u}}
\newcommand{\vecv}{\vec{v}}
\newcommand{\vecw}{\vec{w}}
\newcommand{\vecz}{\vec{z}}

\DeclareMathOperator{\diag}{diag}

\ExplSyntaxOn
\NewDocumentCommand{\cycle}{ O{\;} m }
{
	(
	\alec_cycle:nn { #1 } { #2 }
	)
}

\seq_new:N \l_alec_cycle_seq
\cs_new_protected:Npn \alec_cycle:nn #1 #2
{
	\seq_set_split:Nnn \l_alec_cycle_seq { , } { #2 }
	\seq_use:Nn \l_alec_cycle_seq { #1 }
}
\ExplSyntaxOff

\title[The L'vov-Kaplansky conjecture]{The L'vov-Kaplansky conjecture for polynomials of degree three}

\author{Daniel Vitas}
\address{Department of Mathematics, Faculty of Mathematics and Physics,  University of Ljubljana, Slovenia}
\email{daniel.vitas@gmail.com}


\begin{document}

\begin{abstract} The L'vov-Kaplansky conjecture states that the image of a multilinear noncommutative polynomial $f$ in the matrix algebra $M_n(K)$ is a vector space for every $n \in {\mathbb N}$. We prove this conjecture for the case where $f$ has degree $3$ and $K$ is an algebraically closed field 
of characteristic $0$.
\end{abstract}

\maketitle

\section{Introduction}

Let $K$ be a field. By $K \langle X_1, \ldots, X_m \rangle$ we denote the free algebra in variables $X_1,\ldots, X_m$ over $K$. We call its elements \emph{noncommutative polynomials}. For a noncommutative polynomial $f \in K \langle X_1, \ldots, X_m \rangle$ and a $K$-algebra $A$, we call the set
$$ f(A) = \{ f(a_1,\ldots,a_m) \mid a_1,\ldots,a_m \in A \}$$
the \emph{image} of $f$ in $A$. A noncommutative polynomial $f$ is called \emph{multilinear} if it is of the form
$$ f(X_1,\ldots,X_m) = \sum_{\sigma \in S_m} \lambda_\sigma X_{\sigma(1)} \cdots X_{\sigma(m)}$$
for some $\lambda_\sigma \in K$. 

The image of a noncommutative polynomial is obviously closed under conjugation by invertible elements in $A$. If the polynomial is multilinear, its image is also closed under scalar multiplication. 

The following conjecture  was initiated by Kaplansky and later formulated by L'vov  \cite[Question 1.98]{Dn}.

\begin{conjecture}[L'vov-Kaplansky]
	Let $m \in \N$ and let $K$ be a field. If $f \in K\langle X_1,\ldots,X_m \rangle$ is a multilinear polynomial, then the image of $f$ in $M_n(K)$ is a vector space for every $n \in \N$.
\end{conjecture}

Since the linear span of the image of a multilinear polynomial is a Lie ideal \cite[Theorem 2.3]{BK},  the L'vov-Kaplansky conjecture
can be equivalently stated: the image of a multilinear polynomial $f$ in $M_n(K)$ can only be one of the following four spaces: $\{0\}$, the space of scalar matrices $K$, the space of traceless matrices ${\rm sl}_n(K)$, or the whole matrix algebra $M_n(K)$.

So far, only partial results about the L'vov-Kaplansky conjecture are known. In \cite{KBMR}, Kanel-Belov, Malev, and Rowen proved
this conjecture for $2 \times 2$ matrices (over a quadratically closed field $K$). Even the $n=3$ case remains unsolved, although the aforementioned authors obtained some partial results \cite{KBMR2}.

In \cite{Me}, Mesyan proved that the image of a multilinear polynomial of degree $m=3$ in $M_n(K)$ contains all traceless matrices for every $n \in \N$. The result was later extended to $m=4$ (for $n \geq 3$) by Buzinski and Winstanley \cite{BW}; an error in their paper was corrected in \cite{FMS}. 

Several authors considered the conjecture for algebras $A$ other than $M_n(K)$. In \cite{Ma}, Malev proved the L'vov-Kaplansky conjecture for the algebra of quaternions $A = \HH$. Fagundes proved it for the algebra of strictly upper triangular matrices $A = UT^{(0)}_n$ \cite{Fa}, while Gargate and de Mello proved it for the algebra of upper triangular matrices $A = UT_n$ \cite{GM}. The author proved the conjecture for any algebra $A$ with a surjective inner derivation \cite{Vi}, e.g., the algebra of endomorphisms of an infinite dimensional vector space $A = {\rm End}(V)$, and the $n$-th Weyl algebra $A = A_n$ for every $n \in \N$. In \cite{Sp}, {\v S}penko determined the images of multilienar Lie polynomials of degree at most $4$. Chuang classified the images of (not necessarily multilinear) noncommutative polynomials in finite matrix rings \cite{Ch}. Bre{\v s}ar and {\v S}emrl considered the Waring type problem for matrix algebras \cite{BS,BS2}; see also \cite{BV}.
For more related problems and further references on the L'vov-Kaplansky conjecture, we refer the reader to the survey \cite{RYMKB}.

Our main motivation is the result by Dykema and Klep \cite{DK} which states that for every complex multilinear polynomial $f$ of degree three, the image of $f$ in $M_n(\C)$ is a vector space if either $n$ is even or $n < 17$. The case where $n$ is odd was thus left open, except for small $n$.
The authors mentioned that their result in fact holds for every algebraically closed field of characteristic zero. This is due to the Lefschetz principle (proved by Tarski), which states that for a first-order statement (in the language of rings) there is an $N$, depending on the sentence, such that the sentence is true in the field of complex numbers if and only if the same sentence is true in any algebraically closed field of either characteristic $0$ or characteristic larger than $N$.
We will extend the result of Dykema and Klep to every $n \in \N$, i.e., we will prove the L'vov-Kaplansky conjecture for multilinear polynomials of degree three.

\begin{mainthm}\label{lvov-kaplansky}
	Let $K$ be an algebraically closed field of characteristic $0$. If $f \in K \langle x, y,z \rangle$ is a multilinear polynomial of degree three, then the image of $f$ in $M_n(K)$ is a vector space for every $n \in \N$.
\end{mainthm}

We actually prove a  more general result (Theorem \ref{thm}) in which the assumption on  the field $K$ is less restrictive.

We will now briefly illustrate our approach to proving Theorem \ref{lvov-kaplansky}. More precisely,
for simplicity of exposition we will actually sketch a slight modification of our approach. We should also mention that we use some ideas from \cite{DK}. However, our approach is more conceptual and works for both odd and even integers $n$.
 
Let $K$ be an algebraically closed field and let $f$ be a multilinear polynomial of degree three. Assume that $f$ is not surjective in $M_n(K)$. Since $K$ is algebraically closed and the image 
of $f$ is closed under conjugation by invertible elements, there exists a matrix 
$\aaa \in M_n(K)$ in Jordan normal form
 that does not lie in the image of $f$. Let
$$ \aaa = \ddd + \nnn$$
with a diagonal matrix $\ddd$ and a matrix $\nnn$ with non-zero entries only directly above the main diagonal. Set 
\[
\rrr =
\begin{bmatrix}
	0 & 1 & 0 & \ldots & 0 \\
	0 & 0 & 1 & \ldots & 0 \\
	\vdots & \vdots & \vdots & \ddots & \vdots \\
	0 & 0 & 0 & \ldots & 1 \\
	1 & 0 & 0 & \ldots & 0
\end{bmatrix} \in M_n(K) 
\] and
fix the matrices
\begin{align*}
	\xxx &= \diag(x_1,\ldots,x_n) \rrr^{j} \text{,}\\
	\yyy &= \diag(y_1,\ldots,y_n) \rrr^{k} 
\end{align*}
for some $x_i,y_i \in K$ and $j,k \in \Z$.
Let
\begin{align*}
	\zzz &= \diag(z_1,\ldots,z_n) \rrr^{-j-k} \text{,}\\
	\www &= \diag(w_1,\ldots,w_n) \rrr^{1-j-k}
\end{align*}
for the variables $z_i, w_i \in K$. The system
$$ f(\xxx,\yyy,\zzz+\www) = \aaa$$
is equivalent to
\begin{align}\label{demo}
	f(\xxx,\yyy,\zzz) = \ddd \quad \text{and} \quad f(\xxx,\yyy,\www) = \nnn \text{.}
\end{align}
These are two systems of $n$ equations in $n$ variables. Therefore, we can write \eqref{demo} as
\begin{align} \label{demo2}
		\bbb \begin{bmatrix} z_1 \\ \vdots \\ z_n \end{bmatrix} = \begin{bmatrix} d_1 \\ \vdots \\ d_n \end{bmatrix}
		\quad \mbox{and} \quad
		\ccc \begin{bmatrix} w_n \\ \vdots \\ w_n \end{bmatrix} = \begin{bmatrix} \nu_1 \\ \vdots \\ \nu_n \end{bmatrix} \text{,}
\end{align}
where $d_i$ and $\nu_i$ are the entries of $\ddd$ and $\nnn$, respectively, and $\bbb,\ccc \in M_n(K)$. If both the matrix $\bbb$ and the matrix $\ccc$ are invertible, then the system \eqref{demo2} is solvable, which contradicts $\aaa$ not lying in the image of $f$; so we may assume that
$$ g(x_1,\ldots,x_n,y_1,\ldots,y_n) = (\det\bbb)(\det\ccc) = 0$$
for any choice of $x_i,y_i \in K$. Since the field $K$ is infinite, $g$ is the zero polynomial. Its coefficients are commutative polynomials in the coefficients of the polynomial $f$. This gives us a system of polynomial equations which coefficients of $f$ must solve. Solving this system reduces 
 the proof of Theorem \ref{lvov-kaplansky} to a few special cases. 

We will split the proof into two sections. Section \ref{lemma sec} contains a few lemmas serving to prove the key proposition at the end of the section, which yields several polynomial equations that the coefficients of $f$ must solve. In Section \ref{thm sec} we will first solve the aforementioned system of the polynomial equations, and then use this to prove Theorem \ref{lvov-kaplansky}.

This paper is based on the author's 
Master's Thesis \cite{Vi2}.


\section{Key Proposition} \label{lemma sec}

Let $F$ be an arbitrary field, and let $K$ be an infinite field extending $F$. Throughout this paper, the polynomial coefficients are in $F$, and the matrix entries are in $K$. Let $n \in \N$.
Denote by $\Z_n = \Z / n\Z$ the ring of integers modulo $n$ and write $\Z_n^\times$ for the multiplicative group of all invertible elements in $\Z_n$, i.e., all integers (modulo $n$) coprime to $n$. Similarly, we denote $F^\times = F \setminus \{0\}$ and $K^\times = K \setminus \{0\}$.

Denote by $\vece_i$ for $i \in \Z_n$ the standard basis for $K^n$.
Let $s\colon  K^n \rightarrow K^n$ be the linear map given by
\begin{align*}
	s(\vece_i) = \vece_{i+1} \text{.}
\end{align*}
Of course, this means that $s(\vece_{n-1}) = \vece_0$.

For $x_0,\ldots,x_{n-1} \in K^\times$, consider the matrix
$$\aaa(x_0,\ldots,x_{n-1}) = \begin{bmatrix} x_0 \vecu + x_{1} \vecv & s(x_1 \vecu + x_{2} \vecv) & \ldots & s^{n-1}(x_{n-1}\vecu + x_{0} \vecv) \end{bmatrix}$$
with $\vecu, \vecv \in F^n$. The matrix $\aaa$ will appear in the proof of the key proposition of this section. The goal of Lemmas \ref{osnovna lema}, \ref{circ lema}, and \ref{posledica1} is to prove that either the set of $(x_0,\ldots,x_{n-1})$ for which $\aaa$ is invertible is big or useful information about $\vecu$ and $\vecv$ can be given. Lemma \ref{odcepljena lema} will tell us that in either case, the image of $\aaa$ is at least $n-1$ dimensional and can be controlled to some extent.

\begin{lemma} \label{osnovna lema}
	Let $n \in \N$, let $F$ be any field, and let $K$ be an infinite field extending $F$. Let $\vecu, \vecv \in F^n$, and assume that $\vecu, s \vecu,\ldots,s^{n-1} \vecu$ are linearly independent. If the matrix
	$$\aaa(x_0,\ldots,x_{n-1}) = \begin{bmatrix} x_0 \vecu + x_{1} \vecv & s(x_1 \vecu + x_{2} \vecv) & \ldots & s^{n-1}(x_{n-1}\vecu + x_{0} \vecv) \end{bmatrix}$$
	is non-invertible for all $x_0, \ldots, x_{n-1} \in K^\times$, then there exist an $i \in \Z_n^\times$ and an $n$-th root of unity $\omega \in F$ such that
	\begin{align} \label{l1ii}
		\vecv = -\omega s^i \vecu \text{.}
	\end{align}
\end{lemma}

\begin{proof}
	Denote $\uuu = \begin{bmatrix} \vecu & s \vecu & \ldots & s^{n-1}\vecu \end{bmatrix}$ and $\vvv = \begin{bmatrix} \vecv & s \vecv & \ldots & s^{n-1}\vecv \end{bmatrix}$, and let
	$$g(x_0,\ldots, x_{n-1}) = \det \aaa(x_0,\ldots,x_{n-1}) \text{.}$$
	By assumption, $g$ is zero for any choice of $x_0,\ldots, x_{n-1} \in K^\times$, and as the field $K$ is infinite, $g$ must be the zero polynomial. The coefficient at the term $ x_1^2 x_2 \cdots x_{n-1}$ in the polynomial $g$ is
	\begin{align*}
		\det \begin{bmatrix} \vecv &s\vecu & \ldots &s^{n-1}\vecu \end{bmatrix} = 0 \text{.}
	\end{align*}
	Since $s\vecu, \ldots, s^{n-1} \vecu$ are linearly independent, this shows that
	\begin{align} \label{razvoj v}
		\vecv = \sum_{i=1}^{n-1} \alpha_i s^i \vecu
	\end{align}
	for some $\alpha_i \in F$.
	We additionally let $\alpha_0 = 0$.
	If $\vecv = 0$, then $$g(1,\ldots,1) = \det \uuu = 0 \text{,}$$ which contradicts the linear independence of $\vecu,s\vecu,\ldots,s^{n-1}\vecu$. Hence, there is an $l \in \Z_n$ such that $\alpha_l \neq 0$.
	We will prove that for any choice of different $i_1, \ldots, i_k \in \Z_n$ (for $k \leq n-1$) and every permutation $\sigma \in S_k$ we have
	\begin{align}\label{dolg produkt}
		\alpha_\sigma = \alpha_{i_{\sigma(1)} - i_1} \alpha_{i_{\sigma(2)} - i_2} \cdots \alpha_{i_{\sigma(k)} - i_k} = 0 \text{.}
	\end{align}
	Assume temporarily that \eqref{dolg produkt} holds. First, we want to prove that $l$ is coprime to $n$; for contradiction assume that it is not. Take the smallest positive integer $k \in \Z$ such that $k l = 0$ modulo $n$. Then $i_j = jl$ for $j = 1,\ldots,k$
	are all different. Using the equality \eqref{dolg produkt} for the $k$-cycle $\sigma = \cycle{1, \ldots, k}$ we get
	\begin{align} \label{fix1}
		\alpha_l^{k-1}\alpha_{-(k-1)l} = 0 \text{.}
	\end{align}
	Since $\alpha_l \neq 0$ and $kl = 0$ modulo $n$, we have
	$$\alpha_{-(k-1)l} = \alpha_l = 0 \text{,}$$
	which is a contradiction.
	This shows that $l$ is coprime to $n$. In this case, by the same argument as above, the equation \eqref{fix1} holds for every $k = 1,\ldots,n-1$.
	Since $\alpha_l \neq 0$, this implies
	$$\alpha_{-(k-1)l} = 0$$
	for $k = 1,\ldots, n-1$. The elements $0,-l,-2l,\ldots,-(n-2)l$ are exactly all the integers (modulo n) in $\Z_n$, except $l$, and therefore we have
	\[
	\vecv = \alpha_l s^l \vecu \text{.}
	\]
	The coefficient at the term $x_0 x_1 \ldots x_{n-1}$ in the polynomial $g$ is
	\[
	\det\uuu + \det\vvv  = (1+ (-1)^{l(n-1)}\alpha_l^n)\det\uuu = 0 \text{.}
	\]
	Since $l$ is coprime to $n$, it is easy to see that $(-1)^{l(n-1)} = (-1)^{n-1}$. By assumption, $\vecu,s\vecu,\ldots,s^{n-1}\vecu$ are linearly independent, thus $\det\uuu \neq 0$, and therefore $(-\alpha_l)^n = 1$, which shows \eqref{l1ii}.

	To prove \eqref{dolg produkt} we proceed by induction on $k$. The $k=1$ case is trivial, so assume that $1 < k \leq n-1$ and that \eqref{dolg produkt} holds for all smaller $k$. The coefficient at the term $x_{0+\delta_0}x_{1+\delta_1} \cdots x_{n-1 + \delta_{n-1}}$ in the polynomial $g$ for
	$$\delta_i = \begin{cases} 1,& i \in \{i_1,\ldots,i_k\} \\ 0, & \text{otherwise}\end{cases}$$
	is
	\begin{align} \label{det a zero}
	\det \begin{bmatrix} \veca_0 & s \veca_1 & \ldots & s^{n-1}\veca_{n-1}\end{bmatrix} = 0
	\end{align}
	with
	\[
	\veca_i = \begin{cases} \vecv,& i \in \{i_1,\ldots,i_k\} \\ \vecu, & \text{otherwise}\end{cases} \text{.}
	\]
	By plugging in
	\[
	s^i \vecv = \sum_{j=0}^{n-1} \alpha_{j-i} s^j \vecu \text{,}
	\]
	we see that
	\begin{align*}
		\begin{bmatrix} \veca_0 & s \veca_1 & \ldots & s^{n-1}\veca_{n-1}\end{bmatrix} = \sum_{j_{i_1},\ldots,j_{i_k} = 0}^{n-1} \alpha_{j_{i_1} - i_k} \cdots \alpha_{j_{i_k} - i_k} \det \begin{bmatrix} \vecb_0 & \vecb_1 & \ldots & \vecb_{n-1}\end{bmatrix} 
	\end{align*}
	for
	\[
	\vecb_i = \begin{cases} s^{j_i} \vecu,& i \in \{i_1,\ldots,i_k\} \\ s^{i} \vecu, & \text{otherwise}\end{cases} \text{.}
	\]
	If $\{ j_{i_1}, \ldots, j_{i_k} \} \neq \{ i_1,\ldots, i_k \}$, then the matrix $\begin{bmatrix} \vecb_0 & \vecb_1 & \ldots & \vecb_{n-1}\end{bmatrix}$ contains the same vector twice and thus has determinant zero. If $\{j_{i_1}, \ldots, j_{i_k} \} = \{ i_1,\ldots, i_k \}$, i.e., $j_{i_l} = i_{\pi(l)}$ for some permutation $\pi \in S_k$, then $$\det\begin{bmatrix} \vecb_0 & \vecb_1 & \ldots & \vecb_{n-1}\end{bmatrix} = \pm \det \uuu \text{,}$$
	where the sign depends on the permutation. The equality \eqref{det a zero} thus gives us
	\begin{align*}
		\begin{bmatrix} \veca_0 & s \veca_1 & \ldots & s^{n-1}\veca_{n-1}\end{bmatrix} = \sum_{\pi \in S_k} \pm  \alpha_{i_{\pi(1)}- i_1} \cdots \alpha_{i_{\pi(k)}- i_k} \det\uuu  = 0 \text{,}
	\end{align*}
	and since, by assumption, $\det \uuu \neq 0$, we obtain	
	\begin{align} \label{vsota permutacij}
		\sum_{\pi \in S_k} \pm \alpha_\pi = \sum_{\pi \in S_k} \pm  \alpha_{i_{\pi(1)}- i_1} \cdots \alpha_{i_{\pi(k)}- i_k} = 0 \text{.}
	\end{align}
	If
	$$ \sigma = \cycle{j_{11},\ldots,j_{1{s_1}}}\cdots \cycle{{j_{t1}},\ldots,{j_{t{s_t}}}}$$
	is a decomposition into disjoint cycles, then
	\begin{align} \label{formula za alpha}
		\alpha_\sigma = \left( \alpha_{i_{j_{12}} - i_{j_{11}}} \alpha_{i_{j_{1 3}} - i_{j_{12}} } \cdots \alpha_{i_{j_{11}} - i_{j_{1{s_1}}}}\right) \cdots \left(\alpha_{ i_{j_{t2}} - i_{j_{t1}}} \cdots \alpha_{i_{j_{t1}} - i_{j_{t{s_t}}}}\right) \text{.}
	\end{align}
	This expression is zero if $\sigma$ is not a $k$-cycle. Indeed, consider $i'_l = i_{j_{1l}}$ for $l = 1,\ldots,s_1$, and
	the permutation $$\tau = \cycle{1,2,\ldots,s_1} \in S_{s_1} \text{.}$$
	Since $s_1 < k$ (because $\sigma$ is not a $k$-cycle), by the induction hypothesis we have
	\begin{align*}
		\alpha'_\tau &= \alpha_{i'_2 - i'_1} \alpha_{i'_3 - i'_2} \cdots \alpha_{i'_1 - i'_{k'}}\\
		&= \alpha_{i_{j_{12}} - i_{j_{11}}} \alpha_{i_{j_{13}} - i_{j_{12}}} \cdots \alpha_{i_{j_{11}} - i_{j_{1s_1}}} = 0  \text{.}
	\end{align*}
	The equality \eqref{formula za alpha} then shows that $\alpha_\sigma = 0$ for all $\sigma \in S_k$ that are not $k$-cycles.
	We claim that $\alpha_\sigma=0$ for all $k$-cycles $\sigma$ as well; for contradiction assume that there is a $k$-cycle $\sigma$ such that $\alpha_\sigma \neq 0$. By relabeling the indices, we may assume that $\sigma = \cycle{1, \ldots, k}$. Now take any other $k$-cycle $\rho$ and write it in the form
	$$ \rho = \cycle{1, 2, \ldots, s, j_1, j_2, \ldots, j_{k-s}} \text{,}$$
	where $j_1 \neq s+1$ and $s\geq 1$.
	Consider
	\begin{align*}
		i'_j = \begin{cases} i_j,& j \leq s \\ i_{j_1 + j - s - 1},& j > s \end{cases}
	\end{align*}
	for $j=1,\ldots,k'$, where $k' = k+s-j_1+1$, and the permutation
	$$\tau = \cycle{1,2,\ldots,k'} \in S_{k'}\text{.}$$
	Since $k' < k$, by the induction hypothesis, we have
	\begin{align*}
		\alpha'_\tau &= \alpha_{i'_2 - i'_1} \cdots \alpha_{i'_{s+1} - i'_s} \alpha_{i'_{s+2} - i'_{s+1}} \cdots \alpha_{i'_1 - i'_{k'}}\\
		&= \left( \alpha_{i_2 - i_1} \cdots \alpha_{i_{j_1} - i_s} \right) \left( \alpha_{i_{(j_1 + 1)} - i_{j_1}} \cdots \alpha_{i_1 - i_k } \right) = 0  \text{.}
	\end{align*}
	The second factor cannot be zero as it appears in $\alpha_\sigma$. Thus, $\alpha_\rho = 0$, as the first factor appears in this product. Hence, $\alpha_\pi = 0$ for all $\pi \in S_k$, except $\pi = \sigma$, but this contradicts  \eqref{vsota permutacij}.
\end{proof}

\begin{lemma} \label{circ lema}
	Let $n \in \N$, $n \geq 2$, and let $F$ be a field. Let $$\vecu = \begin{bmatrix} u_1 & u_2 & 0 & \ldots & 0 \end{bmatrix}^\intercal \in F^n \text{,}$$ and assume that the vectors $\vecu, s\vecu, \ldots, s^{n-1} \vecu$ are linearly dependent.
	Then there exists an $n$-th root of unity $\omega \in F$ such that
	\[
	u_1 + \omega u_2 = 0 \text{.}
	\]
\end{lemma}

\begin{proof}
	By assumption, the matrix
	\[
	\uuu = \begin{bmatrix} \vecu & s\vecu & s^2\vecu & \ldots & s^{n-1}\vecu \end{bmatrix}\in M_n(F)
	\]
	is singular, i.e.,
	$$ \det \uuu = u_1^n + (-1)^{n-1} u_2^n = 0 \text{.}$$
	First, assume that $n$ is coprime to the characteristic of $F$.
	Then the polynomial $f(x)=x^n-1$ is separable (since it clearly has no common roots with its derivative), and thus $f$ has $n$ distinct roots in $\ov{F}$, the algebraic closure of $F$. Denote the roots of $f$ by $\omega_1,\ldots,\omega_n$.
	For $k = 1,\ldots,n$ let
	\[
	\vecw_k = \begin{bmatrix} 1 & \omega_k & \omega_k^2 & \ldots & \omega_k^{n-1} \end{bmatrix}^\intercal \in \ov{F}^n \text{.}
	\]
	It is easy to see that $s^{-j} \vecw_k = \omega_k^j \vecw_k$ for any $j \in \Z_n$.
	Hence, we have
	\begin{align*}
		(s^j \vecu)^\intercal \vecw_k = \vecu^\intercal (s^{-j} \vecw_k) 
		= \omega_k^j \vecu^\intercal \vecw_k
	\end{align*}
	for $j \in \Z_n$, and thus $\uuu^\intercal \vecw_k = (\vecu^\intercal \vecw_k) \vecw_k$.
	This tells us that $\vecw_k$ for $k =1,\ldots,n$ are eigenvectors of the matrix
	$\uuu^\intercal$
	for the eigenvalues
	\[
	\lambda_k = \vecu^\intercal \vecw_k \text{.}
	\]
	Since the vectors $\vecw_k$ for $k =1,\ldots,n$ form a Vandermonde matrix and are all different, they are linearly independent. Therefore, $\lambda_k$ are exactly all eigenvalues of the matrix $\uuu^\intercal$. Since, by assumption, $\uuu$ is singular, so is $\uuu^\intercal$, and thus at least one of $\lambda_k$ must be $0$, i.e., for $\omega = \omega_k \in \ov{F}$ we have
	$$ u_1 + \omega u_2 = 0 \text{.}$$
	If $u_2 \neq 0$, then $\omega  = - \frac{u_1}{u_2} \in F$. If $u_2 = 0$, then we can pick $\omega =1$ instead, which again lies in $F$.
	
	Now consider the general case. Let $p$ be the characteristic of $F$. If $p = 0$, we are done, so assume that $p > 0$. Write $n = p^k m$ for a non-negative integer $k$ and an $m \in \N$ coprime to $p$. We can see that
	$$ (u_1^m + (-1)^{m-1} u_2^m)^{p^k} = u_1^n + (-1)^{n - p^k} u_2^n \text{.}$$
	Observe that $(-1)^{p^k} = -1$. Therefore, we have
	$$ \det \uuu = u_1^n + (-1)^{n-1} u_2^n = (u_1^m + (-1)^{m-1} u_2^m)^{p^k} \text{.}$$
	Since $\uuu$ is singular, it follows that
	$$ u_1^m + (-1)^{m-1}u_2^m = 0\text{.}$$
	Because $m$ is coprime to $p$, by the previously proven case, there exists an $m$-th root of unity $\omega \in F$ such that
	$$u_1 + \omega u_2 = 0 \text{.}$$
	Since $m$ divides $n$, $\omega$ is also an $n$-th root of unity, which concludes the proof.
\end{proof}

\begin{lemma} \label{posledica1}
	Let $n \in \N$, $n \geq 3$, let $F$ be any field, and let $K$ be an infinite field extending $F$. Let
	$$\vecu = \begin{bmatrix} u_1 & u_2 & 0 & \ldots & 0 \end{bmatrix}^\intercal \in F^n  \quad\mbox{and}\quad  \vecv = \begin{bmatrix} v_1 & v_2 & 0 & \ldots & 0 \end{bmatrix}^\intercal \in F^n \text{.}$$
	Assume that
	\begin{align} \label{vsota je nic}
	 u_1 + u_2 + v_1 + v_2 = 0 \text{.}
	\end{align}
	If the matrix
	$$\aaa(x_0,\ldots,x_{n-1}) = \begin{bmatrix} x_0 \vecu + x_{1} \vecv & s(x_1 \vecu + x_{2} \vecv) & \ldots & s^{n-1}(x_{n-1}\vecu + x_{0} \vecv) \end{bmatrix}$$
	is non-invertible for all $x_0, \ldots, x_{n-1} \in K^\times$, then one of the following holds:
	\begin{enumerate}[label=(\alph*)]
		\item \label{c1i} There is an $n$-th root of unity $\omega \in F$ such that
		\begin{align*}
			u_1 + \omega u_2 &= 0 \text{,}\\
			v_1 + \omega v_2 &= 0 \text{.}
		\end{align*}
		
		\item \label{c1ii} We have
		\begin{align*}
			u_1 + v_2 &= 0 \text{,} \\
			u_2 = v_1 &= 0 \text{.}
		\end{align*}
	
		\item \label{c1iii} We have
		\begin{align*}
			u_1 = v_2 &= 0 \text{,} \\
			u_2 + v_1 &= 0 \text{.}
		\end{align*}
	\end{enumerate}
\end{lemma}

\begin{proof}
	First, consider the case where $\vecu,s\vecu,\ldots,s^{n-1}\vecu$ are linearly dependent. If $s\vecu, \ldots, s^{n-1}\vecu$ are also linearly dependent, then $\vecu = 0$, and \eqref{vsota je nic} implies the case \ref{c1i} for $\omega=1$; hence, assume that $s\vecu, \ldots, s^{n-1}\vecu$ are linearly independent. By Lemma \ref{circ lema}, there is an $n$-th root of unity $\omega \in F$ such that $$u_1 + \omega u_2 = 0\text{,}$$
	i.e., for the vector $\vecw = \begin{bmatrix} 1 & \omega & \ldots & \omega^{n-1} \end{bmatrix} \in F^n$ we have $$\vecu^\intercal \vecw = 0 \text{.}$$
	Note that the same argument as in the proof of Lemma \ref{osnovna lema} shows that \eqref{razvoj v} holds.
	For $i \in \Z_n$, we have
	$$ (s^i \vecu)^\intercal \vecw = \vecu^\intercal (s^{-i}\vecw) = \omega^i \vecu^\intercal \vecw = 0 \text{,}$$
	and thus \eqref{razvoj v} implies $\vecv^\intercal \vecw = 0$, i.e.,
	$$ v_1 + \omega v_2 = 0 \text{.}$$
	This gives \ref{c1i}.
	
	Now assume that $\vecu,s\vecu,\ldots,s^{n-1}\vecu$ are linearly independent.
	By Lemma \ref{osnovna lema}, there exist an $i \in \Z_n^\times$ and an $n$-th root of unity $\omega \in F$ such that $$\vecv = - \omega s^i \vecu\text{.}$$
	If $i = 1$, then
	\begin{align*}
		\omega u_1 + v_2 &= 0 \text{,} \\
		u_2 = v_1 &= 0 \text{.}
	\end{align*}
	If $i = -1$, then \begin{align*}
		u_1 = v_2 &= 0 \text{,} \\
		\omega u_2 + v_1 &= 0 \text{.}
	\end{align*}
	In all other cases $u_1,u_2,v_1,v_2$ are zero, and each of the previous identities are satisfied. If $\omega = 1$, we have either \ref{c1ii} or \ref{c1iii}; hence, assume that $\omega \neq 1$. Adding the two equations in each case we get
	$$ \omega (u_1 + u_2) + v_1 + v_2 = 0\text{.}$$
	Combining this with \eqref{vsota je nic}, we see that $u_1 + u_2 = 0$ and $v_1 + v_2 = 0$, which is just \ref{c1i} for $\omega = 1$.
\end{proof}

\begin{lemma} \label{odcepljena lema}
	Let $n \in \N$ and let $K$ be an infinite field. Let $\vecu, \vecv \in K^n$, and assume that $s \vecu,\ldots,s^{n-1} \vecu$ are linearly independent. For $x_0,\ldots,x_{n-1} \in K^\times$, let
	$$\aaa(x_0,\ldots,x_{n-1}) = \begin{bmatrix} x_0 \vecu + x_{1} \vecv & s(x_1 \vecu + x_{2} \vecv) & \ldots & s^{n-1}(x_{n-1}\vecu + x_{0} \vecv) \end{bmatrix} \text{.}$$
	Then there is a Zariski open non-empty set $S \subseteq (K^\times)^n$ such that for any $(x_0, \ldots, x_{n-1}) \in S$ the kernel of $\aaa(x_0,\ldots,x_{n-1})$ is either trivial or spanned by a vector with non-zero entries.
\end{lemma}

\begin{remark}
	By symmetry, the lemma also holds if $s \vecv,\ldots,s^{n-1} \vecv$ are linearly independent.
\end{remark}

\begin{proof}
	Let $\uuu = \begin{bmatrix} \vecu & s \vecu & \ldots & s^{n-1}\vecu \end{bmatrix}$ and $\vvv = \begin{bmatrix} \vecv & s \vecv & \ldots & s^{n-1}\vecv \end{bmatrix}$.
	Denote by $\aaa^{ij}$ the matrix obtained from $\aaa$ by omitting the $i$-th row and the $j$-th column, and similarly for the other matrices.
	For the matrix
	$$\bbb = \uuu + \vvv \diag\left(\frac{x_1}{x_0},\frac{x_2}{x_0}, \ldots,\frac{x_0}{x_{n-1}}\right)$$
	we have $\aaa = \bbb \diag(x_0,\ldots,x_{n-1})$, and thus for $x_0,\ldots,x_{n-1} \in K^\times$ the rank of $\aaa$ is the same as the rank of $\bbb$, and $\bbb^{ij}$ is invertible if and only if $\aaa^{ij}$ is invertible.
	Since $s \vecu,\ldots,s^{n-1} \vecu$ are linearly independent, the matrix $\uuu$ without the first column has rank $n-1$. This means that there is an $i \in \Z_n$ such that $\uuu^{i0}$ is invertible. Hence, $\uuu^{i+j,j}$ is invertible for any $j \in \Z_n$.
	We see that
	$$\bbb^{i+j,j} = \uuu^{i+j,j} + \vvv^{i+j,j} \ddd^{jj}$$
	for $\ddd = \diag\left(y_0,\ldots,y_{n-1}\right)$ with $y_i = \frac{x_{i+1}}{x_i}$. Let
	$$S_j = \left\{ (x_0,\ldots,x_{n-1}) \in (K^\times)^n \mid \aaa^{i+j,j} \text{ is invertible}\right\} \text{.}$$
	The sets $S_j$ are obviously Zariski open; we will prove that they are non-empty for any $j \in \Z_n$. By symmetry, we only need to consider the $j=0$ case.
	The polynomial $$g(y_1,\ldots,y_{n-1}) =\det \bbb^{i0}$$ is non-zero as $g(0,\ldots,0) = \det \uuu^{i0} \neq 0$, so there exists a choice of non-zero $y_1,\ldots,y_{n-1} \in K^\times$ such that $\bbb^{i0}$ is invertible (since $K$ is infinite). Thus, for $x_1 = 1$, $x_i = y_1 \cdots y_{i-1}$ for $i = 2,\ldots,n-1$, and $x_0 = y_1 \cdots y_{n-1}$ the matrix $\aaa^{i0}$ is invertible, which proves $S_0 \neq \emptyset$. Since $K$ is infinite, this means that $S = \cap_{j=0}^{n-1} S_j$ is a non-empty Zariski open set which has the desired property. Indeed, for any $(x_0,\ldots,x_{n-1}) \in S$, the matrix $\aaa^{i0}$ is invertible, and thus the kernel of $\aaa$ is at most one-dimensional. If it contained a non-zero vector with $j$-th zero entry, then the matrix $\aaa$ without the $j$-th column would have non-trivial kernel, which cannot hold since $\aaa^{i+j,j}$ is invertible. Hence, the kernel of $\aaa$ is either trivial or spanned by a single vector with non-zero entries.
\end{proof}

We are now in position to prove the key proposition of this section.
Let $K$ be an algebraically closed field and $n \in \N$. For $i,j \in \Z_n$, denote by $\eee_{ij}$ the  matrix in $M_n(K)$, corresponding to the linear map given by $ \vece_k \mapsto \delta_{jk} \vece_i$ for $k \in \Z_n$. Of course, $\eee_{ij}$ are essentially the standard matrix units, but they are indexed by $\Z_n$ rather than the positive integers. 
Define
$$N = \left\{\sum_{i=0}^{n-1} \nu_i \eee_{i,i+1} \, \middle\vert \, \nu_i \in K \text{ such that \textbf{not} only one of $\nu_i$ is non-zero} \right\}$$
and
$$J_n(K) = \left\{\sum_{i=0}^{n-1} d_i \eee_{ii} + \nnn \, \middle\vert \, d_i \in K, \nnn \in N \right\} \subseteq M_n(K) \text{.}$$
For example, this means that $0 \in N$, $\eee_{0,1} \not\in N$, and $\eee_{0,1} + \eee_{1,2} \in N$. The set $J_n(K)$ contains all $n \times n$ Jordan forms, except those that consist of a single $2 \times 2$ Jordan block and a diagonal block, which, as we will later see, are always contained in the image of $f$ (if $n \geq 4$). Hence, if the image of a noncommutative polynomial $f$ in $M_{n}(K)$ contains $J_n(K)$, the polynomial $f$ is surjective in $M_{n}(K)$. The following proposition shows that if the image of $f$ does not contain $J_n(K)$, its coefficients must satisfy many equations.

\begin{proposition} \label{pomembna lema}
	Let $n \in \N$, $n \geq 3$, let $F$ be any field, and let $K$ be an infinite field extending $F$. Let
	\[
	f(x,y,z) = \lambda_1 xyz + \lambda_2 yzx + \lambda_3 zxy + \mu_1 zyx + \mu_2 xzy + \mu_3 yxz
	\]
	for some $\lambda_i, \mu_i \in F$. If
	\[
	J_n(K) \not\subseteq f(M_n(K)) \text{,}
	\]
	then for every permutation $\rho \in \{ {\rm id}, \cycle{1, 2, 3}, \cycle{1, 3, 2}\}$ one of the following holds:
	\begin{enumerate}[label=(\roman*)]
		\item \label{i}
		There exists an $n$-th root of unity $\omega \in F$ such that
		\begin{align*}
			\lambda_{\rho(1)} + \lambda_{\rho(2)} + \omega \lambda_{\rho(3)} &= 0 \text{,}\\
			\mu_{\rho(1)} + \mu_{\rho(2)} + \omega^{-1} \mu_{\rho(3)} &= 0 \text{.}
		\end{align*}
	
		\item
		\label{ii1}
		We have
		\begin{align*}
			\lambda_{\rho(1)} + \lambda_{\rho(2)} + \mu_{\rho(1)} + \mu_{\rho(2)} &= 0 \text{,}\\
			\lambda_{\rho(3)} = \mu_{\rho(3)} &= 0 \text{.}
		\end{align*}
	
		\item
		\label{ii2}
		We have
		\begin{align*}
			\lambda_{\rho(1)} + \lambda_{\rho(2)} &= 0 \text{,}\\
			\mu_{\rho(1)} + \mu_{\rho(2)} &= 0 \text{,} \\
			\lambda_{\rho(3)} + \mu_{\rho(3)} &= 0 \text{.}
		\end{align*}

		\item \label{iii}
		We have
		\begin{align*}
			\lambda_{\rho(1)} = \mu_{\rho(1)} &= 0 \text{,}\\
			\mu_{\rho(2)} + \lambda_{\rho(3)} &= 0 \text{,}\\
			\lambda_{\rho(2)} + \mu_{\rho(3)} &= 0 \text{.}
		\end{align*}
	\end{enumerate}
\end{proposition}

\begin{proof}
	We will only consider the $\rho = {\rm id}$ case; to obtain the other cases one only has to cyclically permute the variables $x,y,z$ (this induces a cyclic permutation of $\lambda_1, \lambda_2, \lambda_3$, and $\mu_1,\mu_2,\mu_3$). For contradiction assume that none of the cases \ref{i}, \ref{ii1}, \ref{ii2}, or \ref{iii} hold. Set
	\begin{align*}
		\xxx &= \sum_{i=0}^{n-1} x_i \eee_{ii} \text{,} \\
		\yyy &= \sum_{i=0}^{n-1} y_i \eee_{i, i + 1} \text{,}\\
		\zzz &= \sum_{i=0}^{n-1} \frac{z_{i-1}}{y_{i-1}} \eee_{i,i-1} \text{,} \\
		\www &= \sum_{i=0}^{n-1} w_{i} \eee_{ii} \text{,}
	\end{align*}
	where $x_i,y_i \in K^\times$ and $z_i, w_i \in K$ have yet to be determined.
	Take an arbitrary matrix
	$$\bbb = \ddd + \nnn \in J_n(K)$$
	with a diagonal matrix $\ddd$ and $\nnn \in N$. The equation $f(\xxx,\yyy, \zzz + \www) = \bbb$ is equivalent to the system of equations
	\begin{align} \label{sistem}
		f(\xxx,\yyy,\zzz) = \ddd  \quad \text{and} \quad f(\xxx,\yyy,\www) = \nnn \text{.}
	\end{align}

	First, we will solve the first equation of the system \eqref{sistem}. We can see that
	\begin{align*}
		f(\xxx,\yyy,\zzz) = \sum_{i=0}^{n-1} z_{i} \Big( &\, x_{i} ((\lambda_1 + \lambda_2) \eee_{ii} + \lambda_3 \eee_{i+1,i+1}) \\
		+ &\, x_{i+1} (\mu_3 \eee_{ii} + (\mu_2 + \mu_1) \eee_{i+1,i+1}) \Big) \text{.}
	\end{align*}
	If, for a vector $\vecb$, we denote by $\vecb_i$ the $i$-th entry of $\vecb$, the above equality can be further written as
	\begin{align}\label{enakost iz katere sledi}
		f(\xxx,\yyy,\zzz) = \sum_{i=0}^{n-1} (\aaa' \vecz)_i \eee_{ii}
	\end{align}
	for $\vecz = \begin{bmatrix} z_0 & z_1 & \ldots & z_{n-1}\end{bmatrix}^\intercal \in K^n$ and
	$$\aaa' = \begin{bmatrix} x_0 \vecu' + x_{1} \vecv' & s(x_1 \vecu' + x_{2} \vecv') & \ldots & s^{n-1}(x_{n-1}\vecu' + x_{0} \vecv') \end{bmatrix} \in M_n(K)$$
	with
	\begin{align*}
		\vecu' &= \begin{bmatrix} (\lambda_1 + \lambda_2) & \lambda_3 & 0 & \ldots & 0 \end{bmatrix}^\intercal \in F^n \text{,} \\
		\vecv' &= \begin{bmatrix} \mu_3 & (\mu_2 + \mu_1) & 0 & \ldots & 0 \end{bmatrix}^\intercal \in F^n \text{.}
	\end{align*}
	For $\ddd = \sum_{i=0}^{n-1} d_i \eee_{ii}$, write $\vecd = \begin{bmatrix} d_0 & d_1 & \ldots & d_{n-1}\end{bmatrix}^\intercal \in K^n$. The equality \eqref{enakost iz katere sledi} shows that solving $f(\xxx,\yyy,\zzz) = \ddd$ is equivalent to solving
	\begin{align} \label{vektorska enacba 1}
		\aaa' \vecz = \vecd \text{.}
	\end{align}
	We can assume that
	$$ \lambda_1 + \lambda_2 + \lambda_3 + \mu_1 + \mu_2 + \mu_3 =  0\text{,}$$
	since otherwise $f$ is surjective, which contradicts $J_n(K) \not\subseteq f(M_n(K))$. The cases \ref{c1i}, \ref{c1ii}, and \ref{c1iii} from Lemma \ref{posledica1} for the matrix $\aaa'$ are exactly the cases \ref{i}, \ref{ii1}, and \ref{ii2}, which by assumption cannot hold. Thus, by Lemma \ref{posledica1}, the set $S'$ of all $(x_0,\ldots,x_{n-1}) \in (K^\times)^n$ such that $\aaa'$ is invertible, is non-empty. Therefore, the equation \eqref{vektorska enacba 1} is solvable for any $(x_0,\ldots, x_{n-1}) \in S'$. Hence, for any choice of $(x_0,\ldots, x_{n-1}) \in S'$ and $y_i \in K^\times$ there exist $z_i \in K$ such that
	$$ f(\xxx,\yyy,\zzz) = \ddd \text{.}$$

	Now we will solve the second equation of the system \eqref{sistem}.
	We can see that
	\begin{align*}
		f(\xxx,\yyy,\www) = \sum_{i=0}^{n-1} \Big( &\, x_{i}  (w_{i+1} \lambda_1 + w_{i} (\mu_2 + \lambda_3)) \\
		+ &\,  x_{i+1} (w_{i} \mu_1 + w_{i+1} (\lambda_2 + \mu_3)) \Big) y_{i}\eee_{i, i+1} \text{,}
	\end{align*}
	which can be written as
	\begin{align}\label{enacba ki implicira 2}
		f(\xxx,\yyy,\www) = \sum_{i=0}^{n-1} (\aaa^\intercal \vecw)_i y_i \eee_{i,i+1}
	\end{align}
	for $\vecw = \begin{bmatrix} w_0 & w_1 & \ldots & w_{n-1}\end{bmatrix}^\intercal \in K^n$ and
	$$\aaa = \begin{bmatrix} x_0 \vecu + x_{1} \vecv & s(x_1 \vecu + x_{2} \vecv) & \ldots & s^{n-1}(x_{n-1}\vecu + x_{0} \vecv) \end{bmatrix} \in M_n(K)$$
	with
	\begin{align*}
		\vecu &= \begin{bmatrix} (\mu_2 + \lambda_3) &  \lambda_1 & 0 & \ldots & 0 \end{bmatrix}^\intercal \in F^n \text{,} \\
		\vecv &= \begin{bmatrix} \mu_1 & (\lambda_2 + \mu_3) & 0 & \ldots & 0 \end{bmatrix}^\intercal \in F^n \text{.}
	\end{align*}
	If $\nnn = 0$, then we can pick $w_i = 0$; hence, assume that $\nnn \neq 0$. Then $\nnn = \sum_{i=0}^{n-1} \nu_i \eee_{i,i+1}$ has at least two non-zero entries. For simplicity assume that $\nu_0,\ldots, \nu_k$ are non-zero and $\nu_{k+1} = \cdots = \nu_{n-1} = 0$, where $1 \leq k \leq n-1$.
	Write $\vec{\nu} = \begin{bmatrix} \nu_0 & \nu_1 & \ldots & \nu_{n-1}\end{bmatrix}^\intercal \in K^n$. The equality \eqref{enacba ki implicira 2} shows that solving $f(\xxx,\yyy,\www) = \nnn$ is equivalent to solving
	\begin{align} \label{vektorska enacba 2}
		\diag(y_0,\ldots,y_{n-1})\aaa^\intercal \vecw = \vec{\nu} \text{.}
	\end{align}
	Since the case \ref{iii} does not hold, either $s\vecu,\ldots,s^{n-1}\vecu$ or $s \vecv, \ldots, s^{n-1} \vecv$ are linearly independent. By Lemma \ref{odcepljena lema} for the matrix $\aaa$, there is a non-empty Zariski open set $S \subseteq (K^\times)^n$ such that for any $(x_0,\ldots,x_{n-1}) \in S$ the kernel of the matrix $\aaa$ is either trivial or spanned by a vector $\vecb \in K^n$ with non-zero entries. The first case is easy to handle, so assume that the second one holds. Pick $\eta_0,\ldots,\eta_k \in K^\times$ such that for the vector $\vec{\eta} = \begin{bmatrix} \eta_0,\ldots,\eta_k, 0,\ldots, 0\end{bmatrix}^\intercal$ we have $\vec{\eta}^\intercal \vecb = 0$ (this can be done since $k \geq 1$ and $\vecb$ has non-zero entries). This means that $\vec{\eta}$ is contained in the image of $\aaa^\intercal$, and therefore the system
	$$ \aaa^\intercal \vecw = \vec{\eta}$$
	can be solved. Taking
	$$ y_i = \begin{cases} \frac{\nu_i}{\eta_i},& i \leq k \\ 1 ,& i > k\end{cases}$$
	solves the system \eqref{vektorska enacba 2}. Hence, for any choice of $(x_0,\ldots,x_{n-1}) \in S$ there are $y_i \in K^\times$ and $w_i \in K$ such that
	$$ f(\xxx,\yyy,\www) = \nnn \text{.}$$

	Observe that the set $S'$ is Zariski open. Since so is $S$ and $K$ is infinite, the intersection of $S$ and $S'$ is non-empty. Thus, we can pick $(x_0,\ldots,x_{n-1}) \in S \cap S'$. Then we can choose $y_i \in K^\times$ and $w_i \in K$ such that $f(\xxx,\yyy,\www) = \nnn$. Finally, we can also pick $z_i \in K$ such that $f(\xxx,\yyy,\zzz) = \ddd$. This solves the system \eqref{sistem}. Hence, $J_n(K) \subseteq f(M_n(K))$, which contradicts the assumption.
\end{proof}



\section{Main Theorem} \label{thm sec}

In this section we will show that no multilinear polynomial of degree three can satisfy the equations given in Proposition \ref{pomembna lema} for every $\rho$ (under an additional assumption). The main theorem will follow from this. 

We first consider a special family of multilinear polynomials.
As usual, we write $[y,z]$ for $yz-zy$.

\begin{lemma} \label{sur1}
	Let $n \in \N$, $n \geq 3$, let $K$ be an infinite field, and let $A$ be a $K$-algebra. If $\lambda \in K$, $\lambda \neq 1$, then the polynomial
	$$f(x,y,z) = x[y,z] - \lambda [y,z]x$$
	is surjective in $M_n(A)$.
\end{lemma}

\begin{proof}
	Let $d_1,\ldots,d_n \in K$ be such that
	$$ d_1 + \dots + d_{n} = 0 \text{.}$$
	Then the matrix $\ddd = \diag(d_1,\ldots,d_{n})$ has trace zero, and by \cite{Al}, there exist $\yyy,\zzz \in M_n(K)$ such that
	$$[\yyy,\zzz] = \ddd \text{.}$$
	If we set $$\xxx = \sum_{i,j} x_{ij} \eee_{ij} \in M_n(A)$$
	for the variables $x_{ij} \in A$, the $ij$-th entry of the matrix $f(\xxx,\yyy,\zzz)$ is
	$$f(\xxx,\yyy,\zzz)_{ij} = (d_j - \lambda d_i) x_{ij} \text{.}$$
	Therefore, it is enough to prove that there exists a choice of $d_1,\ldots,d_{n} \in K$ with $d_1 + \cdots + d_n = 0$ such that
	$$d_j \neq \lambda d_i$$
	for any $i,j = 1,\ldots, n$.
	For contradiction assume that this is not the case. This means that the variety
	$$V = \{ (d_1,\ldots,d_n) \in K^n \mid  d_1 + \cdots + d_n = 0 \}$$
	is contained in the union of the varieties
	$$ V_{ij} = \{ (d_1,\ldots,d_n) \in K^n \mid d_j = \lambda d_i \}$$
	for $i,j = 1,\ldots,n$. Since $K$ is infinite, $V$ is irreducible, and thus we have $V \subseteq V_{ij}$ for some $i,j \in \{1,\ldots,n\}$. But this is clearly not the case since $\lambda \neq 1$ and $n \geq 3$.
\end{proof}

The proof of the following proposition is obtained by solving all the equations appearing in Proposition \ref{pomembna lema}. It only involves elementary yet tedious calculations.

\begin{proposition} \label{prop1}
	Let $n \in \N$, $n \geq 3$, and let $F$ be a field of characteristic not $2$ or $3$ such that for any $n$-th roots of unity $\omega, \eta, \theta \in F$, distinct from $1$, we have
	$$ \omega\eta\theta - \omega -\eta-\theta +2 \neq 0 \text{.}$$
	Let
	\[
	f(x,y,z) = \lambda_1 xyz + \lambda_2 yzx + \lambda_3 zxy + \mu_1 zyx + \mu_2 xzy + \mu_3 yxz
	\]
	for some $\lambda_i, \mu_i \in F$ such that
	\begin{align} \label{vsota ni nic}
		\lambda_1 + \lambda_2 + \lambda_3 \neq 0 \text{.}
	\end{align}
	Then for any infinite field $K$ extending $F$, we have $$J_n(K) \subseteq f(M_n(K))\text{.}$$
\end{proposition}

\begin{proof}
	For contradiction assume that there is an infinite filed $K$ extending $F$ such that
	\begin{align}\label{ne vsebuje}
		J_n(K) \not\subseteq f(M_n(K)) \text{.}
	\end{align}
	By Proposition \ref{pomembna lema}, for every $\rho \in \{ {\rm id}, \cycle{1, 2, 3}, \cycle{1, 3, 2}\}$ one of the cases \ref{i}, \ref{ii1}, \ref{ii2}, or \ref{iii} holds.

	Assume that \ref{iii} holds for some $ \rho \in \{ {\rm id}, \cycle{1, 2, 3}, \cycle{1, 3, 2}\}$ -- without loss of generality, we will assume $\rho = {\rm id}$. Then we have
	\begin{align} \label{caseiiieq1}
		\begin{split}
		\lambda_{1} = \mu_{1} &= 0 \text{,}\\
		\mu_{2} + \lambda_{3} &= 0 \text{,}\\
		\lambda_{2} + \mu_{3} &= 0 \text{.}
		\end{split}
	\end{align}
	By Proposition \ref{pomembna lema}, one of the cases \ref{i}, \ref{ii1}, \ref{ii2}, or \ref{iii} holds for the permutation $\rho = \cycle{1, 3, 2}$.
	
 	If \ref{iii} holds, then we have
	\begin{align*}
		\lambda_{3} = \mu_{3} &= 0 \text{.}
	\end{align*}
	The above equation together with \eqref{caseiiieq1} shows that $\lambda_1,\lambda_2,\lambda_3$ are zero, which contradicts \eqref{vsota ni nic}.
	
	If \ref{ii2} holds, we have
	\begin{align} \label{caseiiieq2}
		\begin{split}
			\lambda_{3} + \lambda_{1} &= 0 \text{,}\\
			\mu_{3} + \mu_{1} &= 0 \text{.}
		\end{split}
	\end{align}
	Since $\lambda_1 = \mu_1 = 0$ by \eqref{caseiiieq1}, the equation \eqref{caseiiieq2} shows that $\lambda_3 = \mu_3 = 0$. Plugging this back into \eqref{caseiiieq1} gives us that $\lambda_1,\lambda_2,\lambda_3$ are zero, which contradicts \eqref{vsota ni nic}.
	
	If \ref{ii1} holds, we have
	\begin{align*}
		\lambda_{2} = \mu_{2} &= 0 \text{.}
	\end{align*}
	Then the equation \eqref{caseiiieq1} shows that $\lambda_1,\lambda_2,\lambda_3$ are zero, which contradicts \eqref{vsota ni nic}.
	
	If \ref{i} holds, there is an $n$-th root of unity $\omega \in F$ such that
	\begin{align*}
		\lambda_{3} + \lambda_{1} + \omega \lambda_{2} &= 0 \text{,}\\
		\mu_{3} + \mu_{1} + \omega^{-1} \mu_{2} &= 0 \text{.}
	\end{align*}
	Since, by the equation \eqref{caseiiieq1}, we have $\lambda_1=\mu_1 = 0$, the above equality shows $\lambda_3 = - \omega \lambda_2$. Plugging this into \eqref{caseiiieq1} gives us
	\begin{align*}
		\mu_2 &= \omega \lambda_2 \text{,} \\
		\mu_3 &= - \lambda_2 \text{.}
	\end{align*}
	This means that
	\begin{align*}
		f(x,y,z) &= \lambda_2 \left(yzx - \omega zxy + \omega xzy - yxz \right) \\
		&= \lambda_2 \left( y[z,x] - \omega [z,x] y \right) \text{.}
	\end{align*}
	By Lemma \ref{sur1}, $f$ is surjective in $M_n(K)$ unless $\lambda_2 = 0$ or $\omega = 1$, but this contradicts either \eqref{ne vsebuje} or \eqref{vsota ni nic}.
	We have thus seen that the case \ref{iii} cannot hold for any cyclic permutation. Therefore, for every $\rho \in \{ {\rm id}, \cycle{1, 2, 3}, \cycle{1, 3, 2}\}$ either \ref{i}, \ref{ii1} or \ref{ii2} holds. We distinguish four possible cases.

	First, assume that the case \ref{i} holds for every $\rho \in \{ {\rm id}, \cycle{1, 2, 3}, \cycle{1, 3, 2}\}$. This means that there exist $n$-th roots of unity $\omega,\eta,\theta \in F$ such that
	\begin{align}\label{caseieq1}
		\begin{split}
			\lambda_1 + \lambda_2 + \omega \lambda_3 &= 0 \text{,}\\
			\lambda_3 + \lambda_1 + \eta \lambda_2 &= 0 \text{,}\\
			\lambda_2 + \lambda_3 + \theta \lambda_1 &= 0 \text{.}
		\end{split}
	\end{align}
	If some of $\omega, \eta,\theta$ were equal $1$, this would contradict \eqref{vsota ni nic}; hence, we can assume that all of them are distinct from $1$. The system of equations \eqref{caseieq1} tells us that the vector $\begin{bmatrix} \lambda_1 & \lambda_2 & \lambda_3 \end{bmatrix}^\intercal \in F^3$ lies in the kernel of the matrix
	$$ \aaa = \begin{bmatrix} 1 & 1 & \omega \\ 1 & \eta & 1 \\ \theta & 1 & 1\end{bmatrix} \in M_3(F) \text{.}$$
	But $\aaa$ is invertible, since
	\begin{align*}
		\det \aaa = \omega + \eta + \theta - \omega\eta\theta-2 \neq 0
	\end{align*}
	by assumption. This shows that $\lambda_1,\lambda_2,\lambda_3$ are zero, which contradicts \eqref{vsota ni nic}.

	Second, assume that the case \ref{i} holds for exactly two permutations $\rho \in \{ {\rm id}, \cycle{1, 2, 3}, \cycle{1, 3, 2}\}$ and does not hold for the third permutation -- without loss of generality, we will assume that \ref{i} holds for $\rho = \cycle{1, 2, 3}$ and $\rho = \cycle{1, 3, 2}$.
	This means there are $n$-th roots of unity $\omega, \eta \in F$, distinct from $1$ (otherwise \eqref{vsota ni nic} is contradicted), such that
	\begin{align*}
		\begin{split}
			\lambda_2 + \lambda_3 + \omega \lambda_1 &= 0 \text{,}\\
			\lambda_3 + \lambda_1 + \eta \lambda_2 &= 0 \text{,}\\
			\mu_2 + \mu_3 + \omega^{-1} \mu_1 & = 0 \text{,}\\
			\mu_3 + \mu_1 + \eta^{-1} \mu_2 &=0 \text{.}
		\end{split}
	\end{align*}
	Solving this system by taking $\lambda_1$ and $\mu_1$ as parameters gives
	\begin{align} \label{caseiveq2}
		\begin{split}
			\lambda_2 &= \frac{1-\omega}{1-\eta} \lambda_1 \text{,}\\
			\lambda_3 &= \frac{\omega \eta - 1}{1-\eta} \lambda_1 \text{,}\\
			\mu_2 &= \frac{1-\omega^{-1}}{1-\eta^{-1}} \mu_1 = \omega^{-1}\eta \frac{1-\omega}{1-\eta} \mu_1\text{,}\\
			\mu_3 &= \frac{\omega^{-1} \eta^{-1} - 1}{1-\eta^{-1}} \mu_1 = \omega^{-1} \frac{\omega\eta- 1}{1-\eta} \mu_1 \text{.}
		\end{split}
	\end{align}
	We must have
	$$\lambda_1 + \lambda_2 + \lambda_3 + \mu_1 + \mu_2 + \mu_3 =(1-\omega)(\lambda_1 -\omega^{-1} \mu_1) = 0 \text{,}$$
	since otherwise the polynomial $f$ is surjective, which contradicts \eqref{ne vsebuje}. If we plug $\mu_1 = \omega \lambda_1$ into \eqref{caseiveq2}, we obtain
	\begin{align} \label{caseiveq3}
		\begin{split}
			\lambda_2 &= \frac{1-\omega}{1-\eta} \lambda_1 \text{,}\\
			\lambda_3 &= \frac{\omega \eta - 1}{1-\eta} \lambda_1 \text{,}\\
			\mu_1 &= \omega \lambda_1 \text{,} \\
			\mu_2 &= \eta \frac{1-\omega}{1-\eta} \lambda_1\text{,}\\
			\mu_3 &= \frac{\omega\eta- 1}{1-\eta} \lambda_1 \text{.}
		\end{split}
	\end{align}
	For $\rho = {\rm id}$ either \ref{ii1} or \ref{ii2} holds. In both cases we have, using the above equality,
	\begin{align*}
		\lambda_3+ \mu_3 = 2 \frac{\omega\eta- 1}{1-\eta} \lambda_1 = 0 \text{.}
	\end{align*}
	Since the characteristic of $F$ is not $2$ and $\lambda_1 \neq 0$ (as $\lambda_1 = 0$  contradicts \eqref{vsota ni nic}), we have $$\eta = \omega^{-1} \text{.}$$
	The equation \eqref{caseiveq3} now tells us that
	\begin{align*}
		\begin{split}
			\lambda_2 &= -\omega \lambda_1 \text{,}\\
			\lambda_3 &= 0 \text{,}\\
			\mu_1 &= \omega \lambda_1 \text{,}\\
			\mu_2 &= -\lambda_1\text{,}\\
			\mu_3 &= 0 \text{,}
		\end{split}
	\end{align*}
	which just means that
	\begin{align*}
		f(x,y,z) &= \lambda_1 \left( xyz - \omega yzx + \omega zyx - xzy \right) \\
		&=  \lambda_1 \left( x[y,z] - \omega [y,z]x \right) \text{.}
	\end{align*}
	By Lemma \ref{sur1}, $f$ is surjective in $M_n(K)$ unless $\lambda_1 = 0$, but this contradicts either \eqref{ne vsebuje} or \eqref{vsota ni nic}.

	Third, assume that the case \ref{i} holds for some $\rho \in \{ {\rm id}, \cycle{1, 2, 3}, \cycle{1, 3, 2}\}$ and does not hold for the other two permutations -- without loss of generality, we will assume $\rho = {\rm id}$. Thus, there exists an $n$-th root of unity $\omega \in F$ such that
	\begin{align}\label{casexieq5}
		\begin{split}
		\lambda_1 + \lambda_2 + \omega \lambda_3 &= 0 \text{,}\\
		\mu_1 + \mu_2 + \omega^{-1} \mu_3 &= 0 \text{.}
		\end{split}
	\end{align}
	For every $\rho \in \{\cycle{1, 2, 3}, \cycle{1, 3, 2}\}$ either \ref{ii1} or \ref{ii2} holds. Since for any $\rho$ both \ref{ii1} and \ref{ii2} imply the equations
	\begin{align*}
		\lambda_{\rho(1)} + \lambda_{\rho(2)} + \mu_{\rho(1)} + \mu_{\rho(2)} &= 0 \text{,} \\
		\lambda_{\rho(3)} + \mu_{\rho(3)} &= 0 \text{,}
	\end{align*}
	in any case we have
	\begin{align}\label{casexieq6}
		\begin{split}
			\lambda_1 + \mu_1 &= 0 \text{,}\\
			\lambda_2 + \mu_2 &= 0 \text{,}\\
			\lambda_3 + \mu_3 &= 0 \text{.}
		\end{split}
	\end{align}
	Combining this with the equation \eqref{casexieq5} gives us
	$$(1-\omega^2) \lambda_3 = 0 \text{.}$$
	If $\lambda_3 = 0$ or $\omega=1$, then the equation \eqref{casexieq5} contradicts \eqref{vsota ni nic}, thus $\omega= -1$ and the equations \eqref{casexieq5} and \eqref{casexieq6} become
	\begin{align}\label{casexieq7}
		\begin{split}
			\lambda_3 &= \lambda_1 + \lambda_2 \text{,}\\
			\mu_1 &= - \lambda_1 \text{,}\\
			\mu_2 &= - \lambda_2 \text{,}\\
			\mu_3 &= - \lambda_1 - \lambda_2 \text{.}\\
		\end{split}
	\end{align}
	Now assume that \ref{ii1} holds for some $\rho \in \{\cycle{1, 2, 3}, \cycle{1, 3, 2}\}$. We will consider the case where $\rho = \cycle{1, 2, 3}$; the other case is similar. We thus have
	\[
		\lambda_1 = \mu_1 = 0 \text{.}
	\]
	By plugging this into \eqref{casexieq7}, we see that
	\begin{align*}
		f(x,y,z) &= \lambda_2 (yzx + zxy - xzy - yxz)\\
		&= \lambda_2 (y [z,x] + [z,x] y) \text{.}
	\end{align*}
	By Lemma \ref{sur1}, $f$ is surjective in $M_n(K)$ unless $\lambda_2 = 0$, but this contradicts either \eqref{ne vsebuje} or \eqref{vsota ni nic}.
	Hence, \ref{ii2} holds for $\rho = \cycle{1, 2, 3}$ and $\rho = \cycle{1, 3, 2}$. Thus, we have
	\begin{align*}
		\lambda_2 + \lambda_3 &= 0 \text{,}\\
		\lambda_1 + \lambda_3 &= 0 \text{.}
	\end{align*}
	By adding these two equations together, we see that
	$$ (\lambda_1 + \lambda_2) + 2 \lambda_3 = 0 \text{.}$$
	Combining this with \eqref{casexieq7}, we obtain
	$$ 3\lambda_3 = 0 \text{.}$$
	Since the characteristic of $F$ is not $3$, this implies that $\lambda_3 = 0$, which contradicts \eqref{vsota ni nic}.

	Fourth, assume that the case \ref{i} does not hold for any permutation in $\{ {\rm id}, \cycle{1, 2, 3}, \cycle{1, 3, 2}\}$. This means that for every $ \rho \in \{ {\rm id}, \cycle{1, 2, 3}, \cycle{1, 3, 2}\}$ either \ref{ii1} or \ref{ii2} holds. As pointed out in the previous case, \eqref{casexieq6} is fulfilled.
	We distinguish three possible cases.
	
	First, assume that for at least two permutations $\rho \in \{ {\rm id}, \cycle{1, 2, 3}, \cycle{1, 3, 2}\}$ the case \ref{ii1} holds -- without loss of generality, we will assume that \ref{ii1} holds for $\rho = {\rm id}$ and $\rho = \cycle{1,3,2}$. Hence, we have
	\begin{align*}
		\lambda_3 = \mu_3 &= 0 \text{,} \\
		\lambda_2 = \mu_2 &= 0 \text{.}
	\end{align*}
	Together with \eqref{casexieq6} this shows that
	\[
		f(x,y,z) = \lambda_1 (xyz-zyx) \text{.}
	\]
	But then \cite[Theorem B]{Kh} shows that $f$ is surjective in $M_n(K)$ unless $\lambda_1 = 0$, which contradicts either \eqref{ne vsebuje} or \eqref{vsota ni nic}.
	
	Second, assume that for some permutation $\rho \in \{ {\rm id}, \cycle{1, 2, 3}, \cycle{1, 3, 2}\}$ the case \ref{ii1} holds, and for the remaining two permutations the case \ref{ii2} holds -- without loss of generality, we will assume $\rho = {\rm id}$. Thus, we have
	\begin{align*}
		\lambda_3 = \mu_3 &= 0 \text{,}
	\end{align*}
	and
	\begin{align*}
		\lambda_2 + \lambda_3 &= 0 \text{,}\\
		\lambda_1 + \lambda_3 &= 0 \text{.}
	\end{align*}
	Plugging $\lambda_3 = 0$ into the last equation shows that $\lambda_1,\lambda_2,\lambda_3$ are zero, which contradicts \eqref{vsota ni nic}.
	
	Third, assume that the case \ref{ii1}  does not hold for any $\rho \in \{ {\rm id}, \cycle{1, 2, 3}, \cycle{1, 3, 2}\}$. This means that for every $\rho \in \{ {\rm id}, \cycle{1, 2, 3}, \cycle{1, 3, 2}\}$ we have \ref{ii2}. Thus,
	\begin{align*}
		\lambda_1 + \lambda_2 &= 0 \text{,}\\
		\lambda_2 + \lambda_3 &= 0 \text{,}\\
		\lambda_3 + \lambda_1 &= 0 \text{.}
	\end{align*}
	Adding all the equations together gives us
	$$ 2 (\lambda_1 + \lambda_2 + \lambda_3) = 0 \text{.}$$
	Since the characteristic of $F$ is not $2$, the above equality contradicts \eqref{vsota ni nic}. This concludes the proof.
\end{proof}

We are now in position to prove the main theorem.

\begin{theorem} \label{thm}
	Let $n \in \N$, $n \geq 4$, and let $F$ be a field of characteristic not $2$ or $3$ such that for any $n$-th roots of unity $\omega, \eta, \theta \in F$, distinct from $1$, we have
	\begin{align} \label{condition}
		\omega\eta\theta - \omega -\eta-\theta +2 \neq 0 \text{.}
	\end{align}
	If $f \in F \langle x,y,z \rangle$ is a multilinear polynomial and $K$ is an algebraically closed extension of $F$, then the image of $f$ in $M_n(K)$ is either ${\rm sl}_n(K)$ or $M_n(K)$.
\end{theorem}

\begin{proof}
	Let
	\[
	f(x,y,z) = \lambda_1 xyz + \lambda_2 yzx + \lambda_3 zxy + \mu_1 zyx + \mu_2 xzy + \mu_3 yxz
	\]
	for some $\lambda_i, \mu_i \in F$.
	If
	\begin{align}\label{vsota je nic 2}
		\lambda_1 + \lambda_2+ \lambda_3 = \mu_1 + \mu_2 + \mu_3 =0 \text{,}
	\end{align}
	then the image of $f$ is obviously contained in ${\rm sl}_{n}(K)$, the space of traceless matrices. By \cite[Theorem 13]{Me}, the image of $f$ contains all traceless matrices and thus,
	$$ f(M_n(K)) = {\rm sl}_{n}(K) \text{.}$$
	Now assume that \eqref{vsota je nic 2} does not hold; without loss of generality assume that
	$$ \lambda_1 + \lambda_2 + \lambda_3 \neq 0 \text{.}$$
	Take an arbitrary matrix $\aaa \in M_n(K)$ in Jordan normal form. If $\aaa$ has either more than one non-trivial Jordan block or a Jordan block of size greater than $2\times 2$, then $\aaa \in J_n(K)$ and Proposition \ref{prop1} implies that $\aaa \in f(M_n(K))$. If $\aaa$ has only one non-trivial Jordan block of size $2 \times 2$ and all other Jordan blocks are of size $1 \times 1$, then we can write
	$$ \aaa = \begin{bmatrix} \bbb & 0 \\ 0 & \ddd \end{bmatrix} \text{,}$$
	where $\bbb \in M_2(K)$ is the non-trivial Jordan block and $\ddd \in M_{n-2}(K)$ is a diagonal matrix. By \cite[Theorem 2]{KBMR}, there exist $\xxx_1, \yyy_1, \zzz_1 \in M_2(K)$ such that $$f(\xxx_1,\yyy_1,\zzz_1) = \bbb \text{.}$$
	Since $n-2 \geq 2$ and any diagonal matrix is contained in $J_{n-2}(K)$, either Proposition \ref{prop1} (if $n \geq 5$) or \cite[Theorem 2]{KBMR} (if $n=4$) show that the image of $f$ in $M_{n-2}(K)$ contains all diagonal matrices. Thus, there are $\xxx_2,\yyy_2,\zzz_2 \in M_{n-2}(K)$ such that $$f(\xxx_2,\yyy_2,\zzz_2) = \ddd\text{.}$$
	By setting
	\begin{align*}
		\xxx &= \begin{bmatrix} \xxx_1 & 0 \\ 0 & \xxx_2 \end{bmatrix} \text{,}\\
		\yyy &= \begin{bmatrix} \yyy_1 & 0 \\ 0 & \yyy_2 \end{bmatrix} \text{,}\\
		\zzz &= \begin{bmatrix} \zzz_1 & 0 \\ 0 & \zzz_2 \end{bmatrix} \text{,}
	\end{align*}
	we get $$f(\xxx,\yyy,\zzz) = \aaa\text{.}$$
	This proves that $f(M_n(K))$ contains all Jordan forms. Since $K$ is algebraically closed and the image of $f$ is closed under conjugation by invertible matrices, this implies that $f$ is surjective in $M_n(K)$.
\end{proof}

\begin{remark} \label{rem2}
		The condition \eqref{condition}  is satisfied by $F=\C$.
		Indeed, we can write
		\begin{align*}
			\omega\eta\theta-\omega&-\eta-\theta+2 =\\
			&(1 -\omega)(1-\eta)(1-\theta)\left( \frac{1}{1-\omega} + \frac{1}{1-\eta} + \frac{1}{1-\theta} - 1 \right) \text{.}
		\end{align*}
		The map $$w \mapsto \frac{1}{1-w}$$
		is a M\"{o}bius transformation which maps the unit circle to the vertical line passing through $\frac{1}{2}$. Thus, the real part of
		$$ \frac{1}{1-\omega} + \frac{1}{1-\eta} + \frac{1}{1-\theta}$$
		is $\frac{3}{2}$ for any $\omega,\eta,\theta \in \C$ from the unit circle without $1$, and thus we have
		$$ \omega\eta\theta-\omega-\eta-\theta+2 \neq 0 \text{.}$$
		Hence, the condition \eqref{condition} holds for $F=\C$ for any $n \in \N$.
		By the Lefschetz principle, for a fixed $n \in \N$, the condition \eqref{condition} is satisfied by any field of either characteristic $0$ or large enough characteristic (depending on $n$).
		
		On the other hand, if $F$ is a field of characteristic $p \geq 5$, the condition \eqref{condition} fails for any $n \in \N$ divisible by $p-1$. Indeed, since $F$ has characteristic $p$, it contains $\Z_p$. Since $\Z_p^\times$ is a group with $p-1$ elements and $p-1$ divides $n$, every non-zero element in $\Z_p$ is an $n$-th root of unity. Pick any $\omega \in \Z_p \setminus \{0,1\}$
		and $\eta \in \Z_p \setminus \{0, 1, \omega^{-1}, 2-\omega\}$ (we can choose such an $\eta$ as $\Z_p$ has at least $5$ elements), and let
		$$\theta = \frac{\omega+\eta-2}{\omega\eta-1} \in \Z_p \text{.}$$
		Since $\eta \neq \omega^{-1}$, $\theta$ is well-defined, and because $\eta \neq 2-\omega$, $\theta$ is non-zero. If $\theta=1$, then
		$$ \omega\eta-\omega-\eta+1 = (\omega-1)(\eta-1) = 0 \text{,}$$
		which cannot hold as $\omega$ and $\eta$ are different from $1$. Hence, $\omega$,$\eta$, and $\theta$ are $n$-th roots of unity, different from $1$, such that
		$$ \omega\eta\theta-\omega-\eta-\theta+2 = 0 \text{.}$$
		The condition \eqref{condition} thus fails in this case. Similar examples also exist for $p=2,3$.
\end{remark}

Finally, we can prove the main theorem.

\begin{theorem}
	Let $K$ be an algebraically closed field of characteristic $0$. If $f \in K \langle x, y, z \rangle$ is a multilinear polynomial of degree three, then the image of $f$ in $M_n(K)$ is a vector space for every $n \in \N$.
\end{theorem}

\begin{proof}
	By the Lefschetz principle, we need only consider the case where $K = \C$.
	Let
	$$f(x,y,z) = \lambda_1 xyz + \lambda_2 yzx + \lambda_3 zxy + \mu_1 zyx + \mu_2 xzy + \mu_3 yxz$$
	for some $\lambda_i,\mu_i \in \C$. By the same argument as in the proof of Theorem \ref{thm}, we can assume that
	$$ \lambda_1 + \lambda_2 + \lambda_3 \neq 0 \text{.}$$
	The case where $n=1$ is trivial. The $n=2$ case is covered in \cite[Theorem 2]{KBMR}, while \cite[Theorem 1.2]{DK} proves the $n=3$ case. Finally, by Remark \ref{rem2}, the condition \eqref{condition} in Theorem \ref{thm} is satisfied, and thus applying Theorem \ref{thm} for $F=\C$ shows
	$$ f(M_n(\C)) = M_n(\C)$$
	for any $n \geq 4$.
	This concludes the proof.
\end{proof}

We also point out a result in the positive characteristic.

\begin{theorem} \label{main cor}
	Let $K$ be an algebraically closed field of characteristic $p \geq 5$, and let $f \in K \langle x,y,z \rangle$ be a multilinear polynomial of degree three. Then there is a $k \in \N$ such that the image of $f$ in $M_n(K)$ is a vector space for any $n \in \N$, $n\geq 4$, coprime to $p^k-1$.
\end{theorem}

\begin{proof}
	Denote by $F$ the extension of $\Z_p$ with the coefficients of $f$. Since $F$ is an extension of $\Z_p$ with finitely many elements, the subfield of $F$ of all algebraic elements is a finite extension of $\Z_p$, and thus has only finitely many, say $p^k$ for some $k \in \N$, elements. This means that every algebraic element in $F^\times$ is a $(p^k-1)$-th root of unity. If $n$ is coprime to $p^k-1$, then the only $n$-th root of unity in $F$ is $1$, and the condition \eqref{condition} in Theorem \ref{thm} trivially holds. Applying Theorem \ref{thm} concludes the proof.
\end{proof}

\section*{Acknowledgment}

The author would like to thank his supervisor Matej Bre\v sar for his valuable review of a draft of the paper, endless patience, and constant guidance. He would also like to thank Daniel Smertnig for his assistance with certain technical aspects of the proof of Theorem \ref{main cor}.

\end{document}